\documentclass[12pt,letterpaper]{article}
\usepackage{amsmath,amsfonts,amsthm,amssymb}

\swapnumbers 

\newtheorem{lemma}{Lemma}[section]
\newtheorem{corollary}[lemma]{Corollary}
\newtheorem{theorem}[lemma]{Theorem}

\theoremstyle{definition} 
\newtheorem{definition}[lemma]{Definition}

\newtheorem{remarks}[lemma]{Remarks}

\newcommand\reals{{\mathbb R}}

\newcommand{\dg}{\sp{\text{\rm o}}}

\begin{document}

\title{Vector lattices in synaptic algebras}

\author{David J. Foulis{\footnote{Emeritus Professor, Department of
Mathematics and Statistics, University of Massachusetts, Amherst,
MA; Postal Address: 1 Sutton Court, Amherst, MA 01002, USA;
foulis@math.umass.edu.}}\hspace{.05 in}, Anna Jen\v cov\'a  and Sylvia
Pulmannov\'{a}{\footnote{ Mathematical Institute, Slovak Academy of
Sciences, \v Stef\'anikova 49, SK-814 73 Bratislava, Slovakia;
pulmann@mat.savba.sk. The second and third authors were supported by
Research and Development Support Agency under the contract No.
APVV-0178-11 and grant VEGA 2/0059/12.}}}

\date{}

\maketitle

\begin{abstract}
A synaptic algebra $A$ is a generalization of the self-adjoint part of
a von Neumann algebra. We study a linear subspace $V$ of $A$ in regard
to the question of when $V$ is a vector lattice. Our main theorem states
that if $V$ contains the identity element of $A$ and is closed under the
formation of both the absolute value and the carrier of its elements, then
$V$ is a vector lattice if and only if the elements of $V$ commute
pairwise.
\end{abstract}


\medskip


\section{Introduction}

Let $A$ be a synaptic algebra \cite{FSyn, PSyn, TDSyn, SymSyn,
ComSyn, 2PSyn, P&ESyn, IdealSyn}. Numerous examples of synaptic algebras
can be found in the cited references. To help fix ideas, one may think
of the special case in which $A$ is the self-adjoint part ${\mathcal B}
\sp{sa}(\mathfrak H)$ of the algebra ${\mathcal B}(\mathfrak H)$ of all
bounded linear operators on a Hilbert space $\mathfrak H$. We shall be
particularly concerned with the mappings on $A$ given by $a\mapsto |a|
$ (the \emph{absolute value} of $a$), $a\mapsto a\sp{+}$ (the \emph{positive
part} of $a$), and $a\mapsto a\dg$ (the \emph{carrier} of $a$). For the case
in which $A={\mathcal B}\sp{sa}(\mathfrak H)$ and $T\in{\mathcal B}\sp{sa}
(\mathfrak H)$, the absolute value of $T$ is $|T|=(T\sp{2})\sp{1/2}$,
$T\sp{+}=\frac12(|T|+T)$, and $T\dg$ is the (orthogonal) projection onto
the closure of the range of $T$.

Among other things, $A$ is a partially ordered linear space. For the
case in which $A={\mathcal B}\sp{sa}(\mathfrak H)$, $\langle\cdot,\cdot
\rangle$ is the inner product on ${\mathfrak H}$, and $S,T\in{\mathcal B}
\sp{sa}(\mathfrak H)$, the partial order is given by $S\leq T
\Leftrightarrow\langle Sx,x\rangle\leq\langle Tx,x\rangle$ for all
$x\in{\mathfrak H}$.

In Section \ref{sc:Prelim}, we provide a brief review of partially
ordered sets and effect algebras that will be needed in what follows.
In Section \ref{sc:BasicProps}, we recall some of the main features
of a synaptic algebra. The notions of a generalized infimum and supremum
are studied in Section \ref{sc:Ginfsup}.

In Section \ref{sc:VecLat} we shall be considering a linear subspace
$V$ of $A$. According to \cite[Theorem 2.8]{ComSyn}, if $V$ is
closed under the mapping $a\mapsto a\sp{+}$, and if $ab=ba$ for all
$a,b\in V$, then $V$ is a vector lattice. In Lemma \ref{lm:CondsforVL}
below we provide more details of the proof of this result and we prove
that, conversely, if $V$ is a vector lattice, the unit $1\in A$ belongs
to $V$, and $V$ is closed under the mappings $a\mapsto a\sp{+}$ and $a
\mapsto a\dg$ then $ab=ba$ for all $a,b\in V$. Theorem \ref{th:TopProp6}
below contains our main results concerning conditions for $V$ to be a
vector lattice.

Many of the results in this paper are motivated by the work of D.
Topping in \cite{TopVL}. Our work differs from Topping's in the
following respects: (1) All of our results pertain to a synaptic algebra
$A$, which considerably generalizes ${\mathcal B}\sp{sa}({\frak H})$.
(2) Our proofs do not require any topological considerations---our
arguments are strictly algebraic and order-theoretic. (3) We make
substantial use of the condition that the linear space $V$ is closed
under the carrier mapping $a\mapsto a\dg$.

\section{Preliminaries} \label{sc:Prelim}

We begin by reviewing some notation, some nomenclature, and some facts
that we shall need later. In what follows, the symbol $:=$ means `equals
by definition,' the abbreviation `iff' means `if and only if,' and
$\reals$ is the ordered field of real numbers.

Let $X$ be a partially ordered set (poset) with partial order relation
$\leq$ and let $x,y\in X$. The elements $x$ and $y$ are said to be
\emph{comparable} iff $x\leq y$ or $y\leq x$. We write $x<y$ to mean
that $x\leq y$ but $x\not=y$. The \emph{infimum} (greatest lower bound)
of $x$ and $y$ in $X$---if it exists---is written as $x\wedge y$, or
alternatively as $x\wedge\sb{X}y$ if it is necessary to make clear where
the infimum is calculated. Likewise, the \emph{supremum} (least upper
bound) of $x$ and $y$ in $X$---if it exists---is written as $x\vee y$ or
as $x\vee\sb{X}y$. We say that the poset $X$ is \emph{lattice ordered},
or simply a \emph{lattice}, if $x\wedge y$ and $x\vee y$ exist for all
$x,y\in X$. A lattice $X$ is \emph{distributive} iff $x\wedge(y\vee z)=
(x\wedge y)\vee(x\wedge z)$, or (as it turns out) equivalently, iff
$x\vee(y\wedge z)=(x\vee y)\wedge(x\vee z)$ for all $x,y,z\in X$.

If there is a smallest element $0\in X$, then the elements $x,y\in X$
are \emph{disjoint} iff $x\wedge y$ exists and $x\wedge y=0$, i.e., iff
for all $z\in X$, $z\leq x,y\Rightarrow z=0$. If there is both a smallest
element $0\in X$ and a largest element $1\in X$, then $X$ is called a
\emph{bounded poset}.  If $X$ is a bounded lattice, then the elements
$x,y\in X$ are \emph{complements} iff $x\wedge y=0$ and $x\vee y=1$, and
the lattice $X$ is \emph{complemented} iff every element $x\in X$ has at
least one complement $y\in X$. An \emph{orthocomplementation} on the
bounded lattice $X$ is a mapping $x\mapsto x\sp{\perp}$ such that, for
all $x,y\in X$, $x$ and $x\sp{\perp}$ are complements, $x=(x\sp{\perp})
\sp{\perp}$, and $x\leq y\Rightarrow y\sp{\perp}\leq x\sp{\perp}$.

An \emph{orthomodular lattice} (OML) \cite{Beran, Kalm} is a bounded
lattice $X$ equipped with an orthocomplementation $x\mapsto x\sp{\perp}$
such that the \emph{orthomodular identity} $x\leq y\Rightarrow y=x\vee
(y\wedge x\sp{\perp})$ holds for all $x,y\in X$. A \emph{Boolean algebra}
is a bounded, complemented, and distributive lattice---or equivalently
it is a distributive OML.

An \emph{MV-algebra} \cite{Chang} is an algebra $(X;0,1,\sp{\perp},+)$
of type (0,0,1,2) such that, for all $x,y\in X$, $x+(y+z)=(x+y)+z$,
$x+y=y+x$, $x+0=x$, $(x\sp{\perp})\sp{\perp}=x$, $0\sp{\perp}=1$,
$x+x\sp{\perp}=1$, and $x+(x+y\sp{\perp})\sp{\perp}=y+(y+x\sp{\perp})
\sp{\perp}$. The MV-algebra $X$ is organized into a poset by defining
$x\leq y\Leftrightarrow y=x+(x+y\sp{\perp})\sp{\perp}$ for $x,y\in X$,
and then $X$ is a bounded distributive lattice with $x\vee y=x+(x+y
\sp{\perp})\sp{\perp}$. Following the traditional mathematical convention,
we usually write the MV-algebra $(X;0,1,\sp{\perp},+)$ simply as $X$.

As is well known, a Boolean algebra is the same thing as an MV-algebra $X$
that satisfies the condition $x+x=x$ for all $x\in X$.

An \emph{effect algebra} \cite{FandB} is a partial algebra $(E;0,1,
\sp{\perp},\oplus)$ of type $(0,0,1,2)$ such that (1) for all $d,e,f
\in E$ the partially defined binary operation $\oplus$, called the
\emph{orthosummation}, satisfies $d\oplus(e\oplus f)=(d\oplus e)
\oplus f$ and $e\oplus f=f\oplus e$, both equalities in the sense that
if either side is defined, then both sides are defined and the equality
holds; moreover it is required that (2) $e\sp{\perp}$, called the \emph
{orthosupplement} of $e$, is the unique element in $E$ for which $e
\oplus e\sp{\perp}$ is defined and $e\oplus e\sp{\perp}=1$; and it is
further required that (3) if $e\oplus 1$ is defined, then $e=0$. In some
accounts, $e\sp{\perp}$ is written as $e\sp{\prime}$.

An effect algebra $(E;0,1,\sp{\perp},\oplus)$ is usually written
simply as $E$. The relation $\perp$, called \emph{orthogonality}, is
defined on $E$ by $e\perp f$ iff the orthosum $e\oplus f$ is defined.
Also, $E$ is organized into a bounded poset by defining $e\leq f$,
called the \emph{induced partial order} on $E$, iff there exists $d
\in E$ such that $e\perp d$ and $e\oplus d=f$. Such an element $d$
is uniquely determined, and by definition $f\ominus e:=d$.  For
$e,f\in E$, we have $(e\sp{\perp})\sp{\perp}=e$ and $e\leq f
\Rightarrow f\sp{\perp}\leq e\sp{\perp}$, whence $e\perp f
\Leftrightarrow e\leq f\sp{\perp}\Leftrightarrow f\leq e\sp{\perp}$.

If $E$ is an effect algebra and $e,f\in E$, then $e$ and $f$ are said
to be \emph{compatible} iff there exist $e\sb{1}, f\sb{1}, d\in E$ such
that $e\sb{1}\perp f\sb{1}$, $(e\sb{1}\oplus f\sb{1})\perp d$, $e=e
\sb{1}\oplus d$ and $f=f\sb{1}\oplus d$. For instance, two comparable
elements of $E$ as well as two orthogonal elements of $E$ are compatible.
Also, if $e$ and $f$ are compatible, then so are $e$ and $f\sp{\perp}$.
A maximal set of pairwise compatible elements of $E$ is called a \emph
{block}. Clearly (Zorn) $E$ is covered by its own blocks.

A subset $S$ of an effect algebra $E$ is called a \emph{sub-effect
algebra} of $E$ iff $0,1\in S$ and for $s,t\in S$, $s\sp{\perp}\in S$
and $s\perp t\Rightarrow s\oplus t\in S$. If $S$ is a sub-effect algebra
of $E$, then $S$ is an effect algebra in its own right under the
restriction to $S$ of the orthosummation and orthosupplementation
operations on $E$, and the resulting induced partial order on $S$ is the
restriction to $S$ of the partial order relation on $E$.

If, the effect algebra $E$ is lattice ordered, i.e., if as a bounded poset
$E$ is a lattice, it is called a \emph{lattice effect algebra}.

\begin{theorem} \label{th:JencaP2.1} {\rm\cite[Proposition 2.1]{JencaMV}}
Suppose that $E$ is a lattice effect algebra. Then the following conditions
are mutually equivalent{\rm: (i)} For all $e,f\in E$, $(e\vee f)\ominus e=f
\ominus(e\wedge f)$. {\rm(ii)} For all $e,f\in E$, $e\wedge f=0\Rightarrow
e\perp f$. {\rm(iii)} For all $e,f\in E$, $e\ominus(e\wedge f)\perp f$.
{\rm(iv)} For all $e,f\in E$, $e$ is compatible with $f$.
\end{theorem}

\begin{definition} \label{df:MV,OML,Bool-EA}
Let $E$ be a lattice effect algebra. Then:
\begin{enumerate}
\item[(1)] If any one, whence all four of the conditions in Theorem
\ref{th:JencaP2.1} hold, then $E$ is called an \emph{MV-effect algebra}.
In particular, $E$ is an MV-effect algebra iff disjoint pairs in $E$
are orthogonal.
\item[(2)] $E$ is an \emph{OML-effect algebra} iff orthogonal pairs
in $E$ are disjoint.
\item[(3)] $E$ is a \emph{Boolean effect algebra} iff disjoint pairs
in $E$ coincide with orthogonal pairs in $E$.
\end{enumerate}
\end{definition}

Every block in a lattice effect algebra $E$ is both a sublattice and a
sub-effect algebra of $E$, and as such it is an MV-effect algebra \cite{ZR}.
Thus, every lattice effect algebra is covered by sub-MV-effect algebras.

An MV-effect algebra $E$ can be organized into an MV-algebra by defining
$e+f:=e\oplus(e\sp{\perp}\wedge f)$ for all $e,f\in E$. Conversely, an
MV-algebra $X$ can be organized into an MV-effect algebra by defining, for
$e,f\in X$, $e\perp f$ iff $e\leq f\sp{\perp}$, in which case $e\oplus f
:=e+f$. In this way, \emph{MV-effect algebras and MV-algebras are
mathematically equivalent structures.}

If $E$ is an OML-effect algebra, then $E$ is an OML with orthocomplementation
$e\mapsto e\sp{\perp}$. Conversely, if $X$ is an OML, then $X$ is organized
into an OML-effect algebra by defining, for $e,f\in X$, $e\perp f$ iff $e
\leq f\sp{\perp}$, in which case $e\oplus f:=e\vee f$. In this way, \emph
{OML-effect algebras and OMLs are mathematically equivalent structures.}

If $E$ is a Boolean effect algebra, then $E$ is a Boolean algebra and for
every $e\in E$, $e\sp{\perp}$ is the unique complement of $e$ in $E$.
Conversely, if $X$ is a Boolean algebra, then $X$ is organized into a
Boolean effect algebra by defining, for $e,f\in X$, $e\perp f$ iff $e\wedge
f=0$, in which case $e\oplus f:=e\vee f$, and by defining $e\sp{\perp}$ to
be the unique complement of $e$ in $E$. In this way, \emph{Boolean
effect algebras and Boolean algebras are mathematically equivalent
structures.}

If an effect-algebra property is attributed to an MV-algebra, an OML, or
a Boolean algebra $X$, it is understood that the property in question holds
when $X$ is organized into an effect algebra as indicated above. For
instance, two elements $e,f$ of an OML $X$ are compatible iff they are
compatible when $X$ is organized into an OML-effect algebra. It is well
known that a Boolean algebra is the same thing as an OML in which every
pair of elements is a compatible pair.

We shall refer to a partially ordered real linear space as a \emph{vector
lattice} iff it is lattice ordered. It is well known that, as a lattice,
a vector lattice is distributive---more generally, a lattice-ordered
group (an {$\ell$}-group) is distributive \cite[Theorem 4]{Birkhoff}.
In an {$\ell$}-group, there are traditional definitions of the positive
and negative parts of an element and of the absolute value of an element
\cite[pages 293 and 295]{Birkhoff}---these, however, may differ from the
corresponding notions for a synaptic algebra. (See Remarks \ref
{rms:traditional} below.)

\section{Basic properties of a synaptic algebra} \label{sc:BasicProps}

In what follows, \emph{we assume that $A$ is a synaptic algebra}.  Axioms
for the synaptic algebra $A$ can be found in \cite{FSyn} and will not be
repeated here, however we shall review some of the important structural
aspects of $A$ that will be needed later in the paper.

Associated with $A$ is a so-called \emph{enveloping algebra} $R\supseteq A$
which is (at least) an associative linear algebra over $\reals$ with a unit
element $1$, and $1\in A$. For the special case in which $A={\mathcal B}
\sp{sa}(\mathfrak H)$, the enveloping algebra is $R={\mathcal B}(\mathfrak H)$.
To avoid trivialities, we assume that $1\not=0$; hence, as is customary, we
may identify each real number $\lambda\in\reals$ with $\lambda1\in R$.

As a subset of $R$, the synaptic algebra $A$ is a real linear space; moreover,
under a partial order relation $\leq$, it forms an \emph{order-unit space}
with order unit $1$ \cite[pp. 67--68]{Alfsen}, i.e., $A$ is an Archimedean
partially ordered real linear space and $1$ is a (strong) order unit in $A$.
Moreover, $A$ is a normed linear space under the \emph{order-unit norm}
defined for $a\in A$ by $\|a\|:=\inf\{0<\lambda\in\reals:-\lambda\leq a
\leq\lambda\}$.

The \emph{positive cone} in $A$ is denoted and defined by $A\sp{+}:=\{a\in A:
0\leq a\}$. If $V$ is a linear subspace of $A$, we understand that $V$ is
organized into a partially ordered real linear space under the restriction
of the partial order $\leq$ on $A$. Thus, if $V$ is a linear subspace of $A$
and $1\in V$, then $V$ is an order-unit space with order unit $1$.

Let $a,b\in A$. The product $ab$ is calculated in $R$ and may or may not
belong to $A$; however, if $a$ commutes with $b$, i.e., $ab=ba$, then $ab
\in A$. If $0\leq a,b$ and $ab=ba$, it turns out that $0\leq ab$. Also,
$a\sp{2}\in A$, whence the \emph{Jordan product} $a\odot b:=\frac12(ab
+ba)=\frac12[(a+b)\sp{2}-(a\sp{2}+b\sp{2})]\in A$.  Let $c:=2(a\odot b)$.
Then $aba=a\odot c-a\sp{2}\odot b$, so $aba\in A$. The mapping $b\mapsto
aba$, called the \emph{quadratic mapping} determined by $a$, is both linear
and order preserving. Moreover, if $0\leq b$, then $aba=0\Rightarrow ab=ba
=0$; in particular, putting $b=1$, we infer that $a\sp{2}=0\Rightarrow a=0$.

For $a,b\in A$, we write $aCb$ iff $ab=ba$, and for $B\subseteq A$, we
define the \emph{commutant} and the \emph{bicommutant} of $B$ by
\[
C(B):=\{a\in A:aCb\text{\ for all\ }b\in B\}\text{\ and\ } CC(B):=
C(C(B)),
\]
respectively. Clearly, $C(B)$ is a linear subspace of $A$ and $1\in C(B)$,
whence, under the restriction of the partial order on $A$, $C(B)$ is an
order-unit space with order unit $1$. For the case in which $B=\{b\}$ we
write $C(b):=C(\{b\})$ and $CC(b):=CC(\{b\})$. If $a\in CC(b)$, we say
that $a$ \emph{double commutes} with $b$.

We shall say that the subset $B\subseteq A$ is \emph{commutative} iff
$aCb$ holds for all $a,b\in B$, i.e., iff $B\subseteq C(B)$. A maximal
commutative subset of $A$ is called a \emph{C-block} (Cf. \cite[Definition
5.1 and Theorem 5.9]{FPMRC}). Evidently, $B\subseteq A$ is a C-block
iff $B=C(B)$. Any commutative subset of $A$, and in particular a
singleton set $\{a\}\subseteq A$, can be enlarged to a C-block (Zorn).

Evidently, for $B, D\subseteq A$, we have (i) $B\subseteq CC(B)$ and (ii)
$B\subseteq D\Rightarrow C(D)\subseteq C(B)$, and it follows from (i) and
(ii) that (iii) $C(CC(B))=C(B)$.

Suppose that $B\subseteq A$ with $B\subseteq C(B)$. Then by (i) and (iii)
above, $CC(B)\subseteq C(B)=C(CC(B))$, whence $CC(B)$ is a commutative
order-unit space with order unit $1$.

The ``unit interval" in $A$ is denoted by $E:=\{e\in A:0\leq e\leq 1\}$,
elements $e\in E$ are called \emph{effects}, and $E$ forms an effect
algebra $(E;0,1,\sp{\perp},\oplus)$ as follows: For $e,f\in E$, $e$ is
orthogonal to $f$, in symbols $e\perp f$, iff $e+f\leq 1$, and then
$e\oplus f:=e+f$. The orthosupplement of $e\in E$ is given by $e
\sp{\perp}:=1-e$ and the induced partial order on $E$ coincides
with the restriction to $E$ of the partial order $\leq$ on $A$. Clearly,
if $e\leq f$, then $f\ominus e=f-e$. It can be shown that $E=\{e\in A:
e\sp{2}\leq e\}$.

An important sub-effect algebra of $E$ is the set $P:=\{p\in A:p\sp{2}=
p\}$ of \emph{projections} in $A$, which turns out to be an OML-effect
algebra. For $p,q\in P$, the infimum and the supremum of $p$ and $q$ in
$P$ are denoted by $p\wedge q$ and $p\vee q$, respectively. Since $P
\subseteq E$, we can regard $p$ and $q$ as effects, and it turns out
that $p\wedge q$ and $p\vee q$ are also the infimum and supremum,
respectively, of $p$ and $q$ in $E$. Therefore, no confusion will
result if we denote an existing infimum or supremum of effects
$e,f\in E$ by $e\wedge f$ or $e\vee f$.

If $p,q\in P$, then it is not difficult to show that $p\perp q
\Leftrightarrow pq=qp=0$, in which case $p\oplus q=p+q=p\vee q$.
Moreover, if $p\leq q$, then $q-p=q\ominus p=q\wedge p\sp{\perp}
\in P$.  Furthermore, if $pCq$, then $p\wedge q=pq$ and $p\vee q
=p+q-pq$.

If $e\in E$ and $p\in P$, then $e\leq p$ iff $e=pep$ iff $e=ep$ iff $e=pe$.
Also, $p\leq e$ iff $p=pe$ iff $p=ep$. Thus, if an effect $e$ and a
projection $p$ are comparable, then $eCp$.

\begin{lemma} \label{lm:compprojs}
Let $p,q\in P$. Then the following conditions are mutually equivalent{\rm:}
{\rm(i)} $p$ is compatible with $q$ in $E$. {\rm(ii)} $pCq$. {\rm(iii)}
$p$ is compatible with $q$ in $P$.
\end{lemma}

\begin{proof}
(i) $\Rightarrow$ (ii). Assume that $p$ is compatible with $q$ in $E$.
Then there exist $e,f,d\in E$ such that $e+f+d\leq 1$, $p=e+d$, and $q=f+d$.
Thus $e+q=e+f+d\leq 1$, so $e\leq 1-q\in P$, whence $e=e(1-q)=e-eq$, and it
follows that $eq=qe=0$. Also, $d\leq f+d=q\in P$, so $dq=qd=d$. Thus, $pq=
(e+d)q=eq+dq=d$. Likewise $qp=q(e+d)=qe+qd=d$, and we have $pq=d=qp$.

(ii) $\Rightarrow$ (iii). Suppose that $pCq$ and put $e:=p-p\wedge q=p-pq$,
$f:=q-p\wedge q=q-pq$, and $d:=p\wedge q=pq$. Clearly, $e,f,d\in P$,
$e+f+d=p+q+pq=p\vee q\leq 1$, $p=e+d$, and $q=f+d$, so $p$ is compatible
with $q$ in $P$. That (iii) $\Rightarrow$ (i) is obvious.
\end{proof}

\begin{lemma} \label{lm:apos&pCa}
If $a\in A\sp{+}$ and $p\in P$, then $pCa\Leftrightarrow 0\leq p\odot a$.
\end{lemma}

\begin{proof}
Assume that $0\leq a\in A$ and $p\in P$. If $pCa$, than since $0\leq p,a$,
we have $0\leq pa=p\odot a$. Conversely, suppose $0\leq p\odot a$. Then
$b:=pa+ap=2(p\odot a)\geq 0$. Clearly, $p\sp{\perp}bp\sp{\perp}=0$, whence
$b-pb=p\sp{\perp}b=0=bp\sp{\perp}=b-bp$, so $pCb$. Therefore, $pa+pap=pb
=bp=pap+ap$, and it follows that $pa=ap$.
\end{proof}

It turns out that $A\sp{+}=\{b\sp{2}:b\in A\}=\{b\sp{2}:b\in A\sp{+}\}$;
moreover, if $a\in A\sp{+}$, there is a unique element---the \emph{square
root} of $a$---denoted by $a\sp{1/2}\in A\sp{+}$ such that $(a\sp{1/2})
\sp{2}=a$. Furthermore, if $a\in A\sp{+}$, then $a\sp{1/2}\in CC(a)$.

\begin{lemma} \label{lm:sqrtprops}
Let $0\leq a,b\in A$. Then{\rm: (i)} $C(a)=C(a\sp{1/2})=C(a\sp{2})$. {\rm(ii)}
$aCb\Leftrightarrow a\sp{1/2}Cb\sp{1/2}$. {\rm(iii)} If $aCb$, then
$a\leq b\Leftrightarrow a\sp{2}\leq b\sp{2}$. {\rm(iv)} If $aCb$ then $a\leq b
\Leftrightarrow a\sp{1/2}\leq b\sp{1/2}$.
\end{lemma}

\begin{proof}
(i) Since $a\sp{1/2}\in CC(a)$, it follows that $C(a)\subseteq C(a\sp{1/2})$.
Obviously, $C(a\sp{1/2})\subseteq C(a)$ and we have $C(a)=C(a\sp{1/2})$.
Replacing $a$ by $a\sp{2}$, we obtain $C(a\sp{2})=C(a)$.

(ii) If $aCb$, then by (i), $b\in C(a)=C(a\sp{1/2})$, whence $a\sp{1/2}
\in C(b)=C(b\sp{1/2})$.  Conversely, if $a\sp{1/2}Cb\sp{1/2}$, then
$(a\sp{1/2})\sp{2}C(b\sp{1/2})\sp{2}$, i.e., $aCb$.

(iii) Part (iii) follows from \cite[Corollary 3.4]{FSyn}.

(iv) Assume that $aCb$. By (ii), $a\sp{1/2}Cb\sp{1/2}$, and replacing
$a$ by $a\sp{1/2}$ and $b$ by $b\sp{1/2}$ in (iii) we obtain (iv).
\end{proof}

For the special case in which $A={\mathcal B}\sp{sa}({\mathfrak H})$
it is well-known that $A$ has the MSR property in the following theorem.
In fact, MSR holds for the self-adjoint part of any C$\sp{\ast}$-algebra
\cite[Proposition 4.2.8]{KadRing}. Also, see \cite[Proposition 10]{TopVL}.
See \cite{JPMSR} for a proof of the theorem.

\begin{theorem} \label{th:MSR}
The synaptic algebra $A$ has the monotone square root {\rm(MSR)}
property, i.e., for all $a,b\in A$, $0\leq a\leq b\Rightarrow
a\sp{1/2}\leq b\sp{1/2}$.
\end{theorem}

For $a\in A$, the \emph{absolute value} of $a$, defined and denoted by
$|a|:=(a\sp{2})\sp{1/2}$, satisfies $|a|\in CC(a)$. Obviously $a\sp{2}=|a|
\sp{2}$ and $|-a|=|a|$. Moreover, if $a,b\in A$ and $aCb$, then $|a|C|b|$
and $|ab|=|a||b|$. The \emph{positive part} of $a\in A$ is denoted and
defined by $a\sp{+}:=\frac12(|a|+a)$, and the positive part of $-a$
(sometimes confusingly called the ``negative part" of $a$) is denoted by
$a\sp{-}:=(-a)\sp{+}=\frac12(|a|-a)$. Then $a\sp{+}, a\sp{-}\in A\sp{+}$;
$a=a\sp{+}-a\sp{-}$; $a\sp{+}a\sp{-}=0$; and $|a|=a\sp{+}+a\sp{-}\geq a\sp{+}
-a\sp{-}=a$. Since $|-a|=|a|$, it follows that $-|a|\leq a\leq |a|$.

\begin{lemma} \label{lm:AbValProps}
Let $a,b\in A$ and suppose that $aCb$. Then{\rm: (i)} $|a+b|\leq |a|+|b|$.
{\rm(ii)} $-b\leq a\leq b\Leftrightarrow |a|\leq b$.
\end{lemma}

\begin{proof} Assume that $aCb$.
(i) We have $a,b\in C(a+b)$, so $a,b\in C(|a+b|)$, and therefore
$|a|,|b|\in C(|a+b|)$, whence $|a+b|C(|a|+|b|)$. Also, $ab\leq |ab|=|a||b|$,
and we have
\[
|a+b|\sp{2}=(a+b)\sp{2}=a\sp{2}+2ab+b\sp{2}\leq a\sp{2}+ 2|a||b|+b\sp{2}
 =(|a|+|b|)\sp{2},
\]
Thus, by Lemma \ref{lm:sqrtprops} (iv), $|a+b|\leq |a|+|b|$.

(ii) Since $aCb$, we have $(b-a)C(b+a)$ and $|a|Cb$. Suppose that $-b
\leq a\leq b$. Then $0\leq b-a, b+a$, whence $0\leq(b-a)+(b+a)=2b$, so
$0\leq b$. Also $0\leq(b-a)(b+a)=b\sp{2}-a\sp{2}$, so $|a|\sp{2}=a\sp{2}
\leq b\sp{2}$, and by Lemma \ref{lm:sqrtprops} (iv), $|a|\leq b$.
Conversely, suppose $|a|\leq b$. Then $a\leq |a|\leq b$ and $-a\leq
|-a|=|a|\leq b$, whence $-b\leq a$ and we have $-b\leq a\leq b$.
\end{proof}

If $a\in A$, there exists a unique projection, denoted by $a\dg\in P$ and
called the \emph{carrier} of $a$, such that, for all $b\in A$, $ab=
0\Leftrightarrow a\dg b=0$.  It turns out that $a\dg\in CC(a)$, $aa\dg=
a\dg a=a$, $a\dg$ is the smallest projection $p\in P$ such that $a=ap$ (or,
equivalently, such that $a=pa$), and for all $b\in A$,
\[
ab=0\Leftrightarrow a\dg b=0\Leftrightarrow ab\dg=0\Leftrightarrow a\dg b
 \dg=0\Leftrightarrow ba=0.
\]
Moreover, $(a\sp{2})\dg=a\dg$, $|a|\dg=a\dg$, and we have the following
lemma.

\begin{lemma} \label{lm:carrierinequalities}
\
\begin{enumerate}
\item If $a,b\in A\sp{+}$, then $(b-a)\dg\leq(b+a)\dg$.
\item If $a,b\in A$ and $-b\leq a\leq b$, then $a\dg\leq b\dg$.
\item If $a,b\in A$ and $0\leq a\leq b$, then $a\dg\leq b\dg$.
\end{enumerate}
\end{lemma}

\begin{proof}
(i) Assume that $a,b\in A\sp{+}$. Then it follows from \cite[Lemma 3.1
(b)]{IdealSyn} that $b\dg, a\dg\leq b\dg\vee a\dg=(b+a)\dg$, whence
$(b-a)(b+a)\dg=b(b+a)\dg-a(b+a)\dg=b-a$, and therefore $(b-a)\dg
\leq(b+a)\dg$.

(ii) and (iii). Let $c:=\frac12(b-a)$ and $d:=\frac12(b+a)$. Then $a
=d-c$, $b=d+c$, and $c,d\in A\sp{+}\Leftrightarrow -b\leq a\leq b$,
whence (ii) follows from (i). Moreover, if $0\leq a\leq b$, then
$-b\leq a\leq b$, so (iii) follows from (ii).
\end{proof}

Using carriers, we can generalize \cite[Lemma 3]{TopVL} to the synaptic
algebra $A$ as follows.

\begin{lemma} \label{lm:TopL3}
Let $a,b,c,d\in A$ with $cd=0$. Then{\rm: (i)} If $0\leq a\leq c$ and
$0\leq b\leq d$, then $ab=ba=0$. {\rm (ii)} If $0\leq a\leq c,d$, then
$a=0$.
\end{lemma}

\begin{proof}
(i) If $a,b,c,$ and $d$ are projections, the statement is easy to prove.
Assume the hypothesis and let $cd=0$. Then $c\dg d\dg=0$, and by Lemma
\ref{lm:carrierinequalities} (iii), $0\leq a\dg\leq c\dg$, and $0\leq b\dg
\leq d\dg$. Therefore $a\dg b\dg=0$, and consequently $ab=ba=0$.

(ii) If $0\leq a\leq c,d$, then $a\sp{2}=0$ by (i), whence $a=0$.
\end{proof}

If $a\in A$, then the \emph{spectral resolution} of $a$ is the
one-parameter family of projections $(p\sb{a,\lambda})\sb
{\lambda\in\reals}$ defined by $p\sb{a,\lambda}:=1-((a-\lambda)
\sp{+})\dg$ for $\lambda\in\reals$. Note that if $V$ is a linear
subspace of $A$, $1\in V$, and $a\in V\Rightarrow a\sp{+}, a\dg
\in V$, then $a\in V\Rightarrow(p\sb{a,\lambda})\sb{\lambda\in
\reals}\subseteq V\cap P$.

\begin{theorem} \label{th:speccom}
Let $a,b\in A$. Then{\rm:}
\begin{enumerate}
\item $aCb$ iff $p\sb{a,\lambda}Cb$ for all $\lambda\in\reals$.
\item $aCb$ iff $p\sb{a,\lambda}Cp\sb{b,\mu}$ for all $\lambda,\mu
\in\reals$.
\end{enumerate}
\end{theorem}

\begin{proof} Part (i) follows from \cite[Theorem 8.10]{FSyn} and
(ii) follows from (i).
\end{proof}

An element $a\in A$ is said to be \emph{invertible} iff there exists
a (necessarily unique) element $a\sp{-1}\in A$ with $aa\sp{-1}=a\sp{-1}a
=1$. It can be shown that $a$ is invertible iff there exists $0<\epsilon
\in\reals$ such that $\epsilon\leq|a|$ and that, if $a$ is invertible,
then $a\sp{-1}\in CC(a)$.

A linear subspace $H$ of $A$ is called a \emph{sub-synaptic algebra}
of $A$ iff $1\in H$ and for all $h\in H$, we have $h\sp{2}, h\dg\in H$,
$0\leq h\Rightarrow h\sp{1/2}\in H$, and $1\leq h\Rightarrow h\sp{-1}
\in H$ \cite[Definition 2.5]{PSyn}. If $H$ is a sub-synaptic algebra
of $A$, then $H$ forms a synaptic algebra in its own right under the
restrictions to $H$ of the operations and the partial order on $A$
\cite[Theorem 2.6]{PSyn}. For instance, the commutant $C(B)$ of any
subset $B\subseteq A$ is a sub-synaptic algebra of $A$, the bicommutant
$CC(B)$ of any commutative subset $B$ of $A$ is a commutative sub-synaptic
algebra of $A$, and any C-block, is a commutative sub-synaptic algebra
of $A$.  Also, the commutant $C(A)$ of $A$ itself, which is called the
\emph{center} of $A$, is clearly commutative, and $C(A)=CC(C(A))$ is a
commutative sub-synaptic algebra of $A$.

\section{The generalized infimum and supremum} \label{sc:Ginfsup}

For the special case in which $A={\mathcal B}\sp{sa}({\mathfrak H})$, a
well-known result of R. Kadison \cite{Kad} shows that $A$ is an \emph
{antilattice}, i.e., for $a,b\in A$, the infimum $a\wedge\sb{A}b$ of $a$
and $b$ in $A$ exists iff $a$ and $b$ are comparable. Thus, in general,
$A$ will fail to be lattice ordered. However, a number of authors (e.g.,
\cite{Stan, LM, TopVL}) have studied the so-called \emph{generalized
infimum} and \emph{generalized supremum} which always exist, which
are defined for $a,b\in A$ by
\[
a\sqcap b:=\frac12(a+b-|a-b|)\text{\ \ and\ \ }a\sqcup b:=\frac12
 (a+b+|a-b|),
\]
respectively, and which in some cases can serve as substitutes for
nonexistent infima and suprema.

All of the properties of the generalized infimum and supremum proved in
\cite[Section II]{Stan} and in \cite[Section 2]{TopVL} generalize to
our synaptic algebra $A$. We consider most of these properties in this
section.

Proofs of the properties of $\sqcap$ and $\sqcup$ in the following
lemma are straightforward, and thus are omitted.

\begin{lemma} \label{lm:TopL1} {\rm Cf. \cite[Lemma 1]{TopVL}}
Let $a,b,c\in A$. Then{\rm:}
\begin{enumerate}
\item $a\sqcap b=b\sqcap a\leq a,b\leq a\sqcup b=b\sqcup a$.
\item $a\leq b\Leftrightarrow a=a\sqcap b\Leftrightarrow b=a\sqcup b$.
\item $a+b=a\sqcap b+a\sqcup b$ {\rm(}Dedekind's law{\rm)} and $|a-b|=
 a\sqcup b-a\sqcap b$.
\item $(a\sqcap b)+c=(a+c)\sqcap(b+c)$ and $(a\sqcup b)+c=(a+c)\sqcup(b+c)$.
\item $-((-a)\sqcap(-b))=a\sqcup b$ and $-((-a)\sqcup(-b))=a\sqcap b$.
\item $a\sqcup 0=a\sp{+}$, $a\sqcap 0=-a\sp{-}$, $a\sqcup(-a)=|a|$, and
 $a\sqcap(-a)=-|a|$.
\item $a\sp{+}a\sp{-}=a\sp{+}\sqcap a\sp{-}=0$.
\item $a-a\sqcap b=(a-b)\sp{+}$ and $b-a\sqcap b=(a-b)\sp{-}$.
\item $(a-a\sqcap b)(b-a\sqcap b)=(a-a\sqcap b)\sqcap(b-a\sqcap b)=0$.
\end{enumerate}
\end{lemma}

Because of \ref{lm:TopL1} (v) we have a ``duality" between $\sqcap$ and
$\sqcup$ enabling us to derive properties of $\sqcup$ from properties
of $\sqcap$ and \emph{vice versa}.

\begin{theorem} \label{th:TopL2} {\rm\cite[Lemma 2]{TopVL}}
For $a,b\in A$, the following conditions are mutually
equivalent{\rm:} {\rm (i)} $a\sqcap b=0$. {\rm(ii)} $a+b=
|a-b|$. {\rm(iii)} $a,b\in A\sp{+}$ and $a\odot b=0$. {\rm(iv)}
$a,b\in A\sp{+}$ and $ab=ba=0$.
\end{theorem}

\begin{proof}
(i) $\Leftrightarrow$ (ii) is obvious.

(ii) $\Rightarrow$ (iii). Assume (ii). Then (i) holds, so $0\leq a,b$
by Lemma \ref{lm:TopL1} (i). Also, $(a+b)\sp{2}=(a-b)\sp{2}$, whence
$ab+ba=-ab-ba$, i.e., $2(ab+ba)=0$, and so $a\odot b=0$.

(iii) $\Rightarrow$ (iv). Assume (iii). Then $ab+ba=0$, so $ab=-ba$,
and we have
\[
a\sp{2}b\sp{2}=a(ab)b=a(-ba)b=-abab
\]
\[
=-(-ba)(-ba)=-b(ab)a=-b(-ba)a=b\sp{2}a\sp{2},
\]
whence $a\sp{2}Cb\sp{2}$. Since $0\leq a,b$, it follows from Lemma
\ref{lm:sqrtprops} (ii) that $aCb$, and therefore $ab=-ab$, so
$ab=ba=0$.

(iv) $\Rightarrow$ (ii). Assume that $0\leq a,b$ and $ab=ba=0$. Then
$aCb$, so $(a+b)C|a-b|$. Also, $(a+b)\sp{2}=a\sp{2}+b\sp{2}=(a-b)
\sp{2}=|a-b|\sp{2}$, whence $a+b=|a-b|$ by Lemma \ref{lm:sqrtprops}
(iv).
\end{proof}

\begin{corollary} {\rm\cite[Corollary 1]{TopVL}}
For $a,b\in A$, the following conditions are mutually equivalent
{\rm: (i)} $|a|\sqcap |b|=0$. {\rm(ii)} $|a||b|=0$. {\rm(iii)}
$ab=0$. {\rm(iv)} $a\sp{+}b\sp{+}=a\sp{+}b\sp{-}=a\sp{-}b\sp{+}
=a\sp{-}b\sp{-}=0$.
\end{corollary}

\begin{proof}
(i) $\Leftrightarrow$ (ii) $\Leftrightarrow$ (iii). By Theorem
\ref{th:TopL2} $|a|\sqcap|b|=0\Leftrightarrow|a||b|=0$, and (ii)
$\Leftrightarrow$ (iii) follows immediately from the facts that
$|a|\dg=a\dg$ and $|b|\dg=b\dg$.

(ii) $\Leftrightarrow$ (iv). Suppose that $|a||b|=0$. Since $0
\leq a\sp{+}, a\sp{-}\leq a\sp{+}+a\sp{-}=|a|$ and likewise $0
\leq b\sp{+}, b\sp{-}\leq |b|$, Lemma \ref{lm:TopL3} (i) yields
(iv). That (iv) $\Rightarrow$ (ii) is obvious.
\end{proof}

Part (i) of the next lemma is a generalization to the synaptic
algebra $A$ of a remark of S. Gudder \cite[p. 2639]{Stan} and part
(ii) corresponds to \cite[Corollary 2]{TopVL}.

\begin{lemma} \label{lm:TopC2}
Let $a,b\in A\sp{+}$. Then{\rm: (i)} $aCb\Rightarrow 0\leq a\sqcap b$.
{\rm (ii)} $a\sqcap b\leq 0\Rightarrow a\sqcap b=ab=ba=0$.
\end{lemma}

\begin{proof}
Assume that $0\leq a,b$.

(i) Suppose that $aCb$. Then $0\leq ab=ba$, whence $0\leq 2(ab+ba)$, so
$-ab-ba\leq ab+ba$, and it follows that $|a-b|\sp{2}=(a-b)\sp{2}\leq(a+b)
\sp{2}$. As $|a-b|C(a+b)$, Lemma \ref{lm:sqrtprops} (iv) yields $|a-b|
\leq a+b$, and therefore $0\leq\frac12(a+b-|a-b|)=a\sqcap b$.

(ii) Suppose that $a\sqcap b\leq 0$. Then $0\leq a\leq a-a\sqcap b$ and
$0\leq b\leq b-a\sqcap b$, whence $ab=ba=0$ by Lemma \ref{lm:TopL1} (ix)
and Lemma \ref{lm:TopL3} (i), so $a\sqcap b=0$ by Theorem \ref{th:TopL2}.
\end{proof}

\begin{lemma} \label{lm:TL4} {\rm\cite[Lemma 4]{TopVL}}
For $a,b\in A$, $aCb\Leftrightarrow a,b\in C(a\sqcap b)$.
\end{lemma}

\begin{proof}
Put $c:=a\sqcap b$. Then by Lemma \ref{lm:TopL1} (ix), $(a-c)(b-c)=
ab-ac-cb+c\sp{2}=0$ and $(b-c)(a-c)=ba-bc-ca+c\sp{2}=0$, whence $ab=
cb+ac-c\sp{2}$ and $ba=bc+ca-c\sp{2}$, so if $cb=bc$ and $ac=ca$, it
follows that $ab=ba$. The converse is obvious.
\end{proof}

\begin{lemma} \label{lm:TC3} {\rm{\cite[Corollary 3]{TopVL}}}
If $a,b\in A$, and $a+b\in A\sp{+}$, then $0\leq a\odot b
\Rightarrow 0\leq a\sqcap b$.
\end{lemma}

\begin{proof}
Suppose that $a+b\in A\sp{+}$, and $0\leq ab+ba$. Then $-ab-ba
\leq ab+ba$, whence $|a-b|\sp{2}=a\sp{2}-ab-ba+b\sp{2}\leq a\sp{2}
+ab+ba+b\sp{2}=(a+b)\sp{2}$ with $0\leq a+b$. Then by MSR, $|a-b|
\leq a+b$, and it follows that $0\leq a\sqcap b$.
\end{proof}

The next theorem strengthens \cite[Corollary 2.4 (a)]{Stan}.

\begin{theorem} \label{th:SC2.4}
If $a,b\in A$, then $a\sqcap b$ is a maximal lower bound in $A$ for $a$ and $b$.
\end{theorem}

\begin{proof} By Lemma \ref{lm:TopL1} (i), $a\sqcap b\leq a,b$. Suppose $c\in A$
with $a\sqcap b\leq c\leq a,b$. Then
\[
0\leq c-a\sqcap b\leq a-a\sqcap b, b-a\sqcap b.
\]
By Lemma \ref{lm:TopL1} (ix), $(a-a\sqcap b)(b-a\sqcap b)=0$, whence by
Lemma \ref{lm:TopL3} (ii), $c-a\sqcap b=0$, i.e., $c=a\sqcap b$.
\end{proof}

Using the results above, properties of the generalized infimum and supremum
for effects and projections obtained in \cite[Corollaries 2.4, 2.5, 2.6, 2.7]
{Stan} and \cite[Proposition 1 and Corollary 3]{TopVL} are easily extended
to the synaptic algebra $A$.

The following lemma is amusing, but for our present purposes, it is only a
curiosity. Topping calls self-adjoint operators $a$ and $b$ \emph{disjunctive}
\cite[p. 17]{TopVL} iff they satisfy the condition in the hypothesis of part
(i) of the lemma. Topping's proof goes through for the synaptic algebra $A$
and will not be repeated here.

\begin{lemma} \label{lm:TopP21} Cf. \rm{{\cite[Proposition 2]{TopVL}}}
Let $a,b,p\in A$. Then{\rm:}
\begin{enumerate}
\item $|a-b|\sqcap(a\sqcap b)=0\Rightarrow0\leq a\sqcap b,a,b$ and
 $ab=ba=(a\sqcap b)\sp{2}$.
\item $|1-p|\sqcap(1\sqcap p)=0\Leftrightarrow p\in P$.
\end{enumerate}
\end{lemma}

\section{Vector lattices in $A$} \label{sc:VecLat}

In this section \emph{we assume that $V$ is a linear subspace of $A$}.
Thus, as mentioned above, $V$ is a partially ordered real linear
space with positive cone $V\sp{+}:=V\cap A\sp{+}$ under the restriction
of the relation $\leq$ on $A$, and if $1\in V$, then $V$ is an order-unit
normed space. In this section we shall obtain analogues for $V\subseteq A$
of Topping's results for $V\subseteq{\mathcal B}\sp{sa}({\frak H})$ in
\cite[\S 3]{TopVL}.

\begin{remarks} \label{rm:PropsofV}
We shall be especially interested in cases in which $V$ has one or more
of the following properties: (i) $1\in V$, (ii) $a\in V\Rightarrow
a\sp{2}\in V$, whence $V$ is closed under Jordan products;  (iii) $a
\in V\Rightarrow |a|\in V$; (iv) $a\in V\Rightarrow a\dg\in V$; (v) $V$
is commutative; (vi) $V$ is a vector lattice. In \cite[\S 3]{TopVL}, Topping
refers to conditions (iii) and (vi) as properties A and L, respectively.
\end{remarks}

\begin{remarks}
There are four very simple ways to form linear subspaces $V$ of $A$ with
properties (i)--(v) in Remarks \ref{rm:PropsofV}, namely:
\begin{enumerate}
\item[(1)] Take any commutative subset $B$ of $A$ (e.g., $B=\{a\}$ for any
$a\in A$) and form $V:=CC(B)$. Then $B\subseteq V$.
\item[(2)] Take any commutative subset $B$ of $A$ and extend $B$ to a
C-block $V$ (Zorn). Then $B\subseteq V$.
\item[(3)] Take any block $B$ in $P$ and form $V=C(B)$. Then if $a\in A$,
it follows that $a\in V$ iff, for all $\lambda\in\reals$, every projection
$p\sb{a,\lambda}$, in the spectral resolution of $a$ belongs to $B$.
Moreover, $V$ is a C-block and $B\subseteq V$.
\item[(4)] Take any element $a\in A$, expand the commutative set $\{p
\sb{a,\lambda}:\lambda\in\reals\}$ of all projections in the spectral
resolution of $a$ to a block $B\sb{a}$ in $P$ (Zorn), and let $V:=
C(B\sb{a})$ as in (3). Then $V$ is a C-block and $a\in B\sb{a}\subseteq
V$.
\end{enumerate}
In all four cases (1)--(4), $V$ is a commutative sub-synaptic algebra
of $A$.
\end{remarks}
In our subsequent work, we use the next lemma routinely and without
explicit attribution.

\begin{lemma} \label{lm:absvalinV}
The following conditions are mutually equivalent{\rm: (i)} $a\in V
\Rightarrow|a|\in V$. {\rm(ii)} $a\in V\Rightarrow a\sp{+}\in V$. {\rm (iii)}
$a\in V\Rightarrow a\sp{-}\in V$. {\rm (iv)} $a,b\in V\Rightarrow a\sqcap b
\in V$. {\rm (v)} $a,b\in V\Rightarrow a\sqcup b\in V$.
\end{lemma}

\begin{proof}
The proof follows easily from the facts that $a\sp{+}=\frac12(|a|+a)$, $a\sp{-}=
(-a)\sp{+}$, $a\sqcap b=\frac12(a+b-|a-b|)$, and $a\sqcup b=\frac12(a+b+|a-b|)$.
\end{proof}

\begin{theorem} \label{th:Topveclat} {\rm{Cf. \cite[Proposition 3]
{TopVL}}}
Suppose that $V$ is a vector lattice. Then{\rm: (i)} If $a\in V\Rightarrow|a|
\in V$, then for all $a,b\in V$, $a\sqcap b=a\wedge\sb{V}b\in V$ and $a\sqcup
b=a\vee\sb{V}b\in V$. {\rm(ii)} The condition $a\in V\Rightarrow|a|\in V$
is equivalent to $a\in V \Rightarrow |a|=a\vee\sb{V}(-a)$.
\end{theorem}

\begin{proof} Let $V$ be a vector lattice.

(i) Assume that $a\in V\Rightarrow|a|\in V$ and $a,b\in V$. Then by
Theorem \ref{th:SC2.4} (ii), $a\sqcap b$ is a maximal lower bound for
$a$ and $b$ in $A$. Since $a\sqcap b\in V$ and $a\sqcap b\leq a,b$, we
have $a\sqcap b\leq a\wedge\sb{V}b\leq a,b$, whence $a\sqcap b=a\wedge
\sb{V}b$. Because $a\mapsto -a$ is an order-reversing linear automorphism
of $V$, it follows that $a\vee\sb{V}b=-((-a)\wedge\sb{V}(-b))=-((-a)
\sqcap(-b))=a\sqcup b$.

(ii) If $a\in V\Rightarrow|a|\in V$ and $a\in V$, then by (i), $a\vee
\sb{V}(-a)=a\sqcup(-a)=\frac12(0+|a+a|)=|a|$. The converse is obvious.
\end{proof}

\begin{remarks} \label{rms:traditional}
If $V$ is a vector lattice, then as mentioned above, there are traditional
definitions of the absolute value, the positive part, and the negative
part of an element $a\in V$; however these notions need not coincide with
the synaptic-algebra definitions of $|a|$, $a\sp{+}$, and $a\sp{-}$, which
may not even belong to $V$. However, by Theorem \ref{th:Topveclat}, if
$a\in V\Rightarrow|a|\in V$, then the traditional notions do in fact
agree with the corresponding synaptic-algebra notions.
\end{remarks}

In \cite[\S 3]{TopVL}, condition (iii) in the following lemma is
called property J.

\begin{lemma} \label{lm:J}
Suppose that $1\in V$ and $a\in V\Rightarrow |a|,\,a\dg\in V$. Then the
following conditions are equivalent{\rm: (i)} $V$ is commutative.
{\rm(ii)} $P\cap V$ is commutative. {\rm(iii)} $a,b\in V\sp{+}
\Rightarrow 0\leq a\odot b$.
\end{lemma}

\begin{proof}
(i) $\Leftrightarrow$ (ii). Trivially, (i) $\Rightarrow$ (ii).
Conversely, assume (ii) and let $a,b\in V$. If $\lambda\in\reals$,
then since $V$ is a linear subspace of $A$, $1\in V$, and $V$ is
closed under the formation of positive parts and carriers, we have
$p\sb{a,\lambda}=1-((a-\lambda)\sp{+})\dg\in V$, whence the spectral
resolution $(p\sb{a,\lambda})\sb{\lambda\in\reals}$ is contained in
$P\cap V$. Likewise, the spectral resolution $(p\sb{b,\lambda})\sb
{\lambda\in\reals}$ is contained in $P\cap V$, and since $P\cap V$
is a commutative set, it follows $p\sb{a,\lambda}Cp\sb{b,\mu}$ for
all $\lambda,\mu\in\reals$. Consequently, $aCb$ by Theorem \ref
{th:speccom} (ii), proving (i).

(i) $\Rightarrow$ (iii). If $aCb$, then $a\odot b=ab$, so (i)
$\Rightarrow$ (iii) is obvious.

(iii) $\Rightarrow$ (ii). Assume (iii) and let $p,q\in V\cap P$.
Then $0\leq p\odot q$, whence $pCq$ by Lemma \ref{lm:apos&pCa}.
\end{proof}

In \cite[\S 3]{TopVL}, Topping refers to condition (ii) in the
next theorem as property R and remarks that the implication
(i) $\Rightarrow$ (ii) is ``classical."

\begin{theorem}  {\rm{Cf. \cite[Lemmas 6 and 10]{TopVL}}}
\label{th:PropR}
Consider the following conditions{\rm:}
\begin{enumerate}
\item $V$ is a vector lattice.
\item If $a,b,c\in V\sp{+}$ and $c\leq a+b$, then there exist
 $a\sb{1}, b\sb{1}\in V\sp{+}$ with $a\sb{1}\leq a$, $b\sb{1}\leq b$,
 and $c=a\sb{1}+b\sb{1}$ {\rm(the Riesz Decomposition Property (RDP))}.
\item If $a,b,c,d\in V$ and $a,b\leq c,d$, then there exists $g\in V$
 with $a,b\leq g\leq c,d$ {\rm(the Riesz Interpolation Property)}.
\item If $a,b\in V\sp{+}$ and $ab=0$, then $a\wedge\sb{V}b$ exists and
 $a\wedge\sb{V}b=0$.
\end{enumerate}
Then {\rm(i)} $\Rightarrow$ {\rm(ii)} $\Leftrightarrow$ {\rm(iii)}
$\Rightarrow$ {\rm(iv)}. Moreover, if $a\in V\Rightarrow|a|\in V$,
then {\rm(i)}, {\rm(ii)}, {\rm(iii)}, and {\rm(iv)} are mutually
equivalent.
\end{theorem}

\begin{proof}
(i) $\Rightarrow$ (ii) $\Leftrightarrow$ (iii). A proof that (i)
$\Rightarrow$ (ii) is easily obtained by putting $a\sb{1}:=c\wedge
\sb{V}a$ and $b\sb{1}:=c-a\sb{1}$. For a proof (in the context of
partially ordered abelian groups, but which works here too) that
(ii) $\Leftrightarrow$ (iii), see \cite[Proposition 2.1 (a) and (b)]
{Good}.

(iii) $\Rightarrow$ (iv). Assume (iii) and $0\leq a,b$ with $ab=0$.
Suppose that $c\in V$ with $c\leq a,b$. Then $0,c\leq a,b$, so by
(iii) there exists $g\in V$ with $0,c\leq g\leq a,b$. Thus, by
Lemma \ref{lm:TopL3} (ii), $c\leq g=0$, and it follows that
$a\wedge\sb{V}b=0$.

Finally, assume that $a\in V\Rightarrow|a|\in V$.

(iv) $\Rightarrow$ (i). Assume (iv), let $c,d\in V$, and put $a:=c-d$.
Since $0\leq a\sp{+}, a\sp{-}\in V$ and $a\sp{+}a\sp{-}=0$, we have
$a\sp{+}\wedge \sb{V}a\sp{-}=0$ by (iv), whence
\[
-a\sp{-}=0-a\sp{-}=(a\sp{+}\wedge\sb{V}a\sp{-})-a\sp{-}=(a\sp{+}-a\sp{-})
\wedge\sb{V}(a\sp{-}-a\sp{-})=a\wedge\sb{V}0.
\]
Thus, $d-a\sp{-}=(a\wedge\sb{V}0)+d=(a+d)\wedge\sb
{V}d=c\wedge\sb{V}d$. Replacing $c$ and $d$ by $-c$ and $-d$, we also
obtain $-((-c)\wedge\sb{V}(-d))=c\vee\sb{V}d$, so (i) holds.
\end{proof}

In \cite[\S 3]{TopVL}, conditions (i) and (ii) in the next lemma are
called properties O (Ogasawara's condition) and J$\sp{+}$, respectively.
The proof is adapted from \cite[Lemmas 6 and 11]{TopVL}.

\begin{lemma} \label{lm:O&Jplus}
Consider the following conditions{\rm:}
\begin{enumerate}
\item If $a,b\in V\sp{+}$, then $a\leq b\Rightarrow a\sp{2}\leq b\sp{2}$.
\item If $a,b\in V\sp{+}$, then $a\leq b\Rightarrow 0\leq a\odot b$.
\item If $a,b\in V\sp{+}$ and $0\leq a\odot b$, then $0\leq a\sqcap b$.
\item If $a,b\in V\sp{+}$ and $ab=0$, then $a\wedge\sb{V}b$ exists
 and $a\wedge\sb{V}b=0$.
\item $V$ is a vector lattice.
\end{enumerate}
Then {\rm(i)} $\Leftrightarrow$ {\rm(ii);}\ \ {\rm(i)} $\Rightarrow$
{\rm(iii);}\ \ {\rm(i)} $\Rightarrow$ {\rm[(ii)} and {\rm(iii)]} $
\Rightarrow$ {\rm(iv);}\ \ and if $a\in V\Rightarrow|a|\in V$, then
{\rm(iv)} $\Rightarrow$ {\rm(v)}.
\end{lemma}

\begin{proof}
(i) $\Rightarrow$ (ii). Assume (i) and $0\leq a\leq b$. Then $0\leq
b-a\leq b+a$, whence $(b-a)\sp{2}\leq(b+a)\sp{2}$, and it follows that
$a\odot b=\frac14[(b+a)\sp{2}-(b-a)\sp{2}]\geq 0$.

(ii) $\Rightarrow$ (i). Assume (ii) and $0\leq a\leq b$. Then $0\leq
b-a\leq b+a$, whence $0\leq(b-a)\odot(b+a)=b\sp{2}-a\sp{2}$.

\smallskip

(i) $\Rightarrow$ (iii). If (i) holds, then (iii) follows from Lemma
\ref{lm:TC3}.

\smallskip

[(ii) and (iii)] $\Rightarrow$ (iv). Assume (ii), (iii), and $a,b\in
V\sp{+}$ with $ab=0$. Also, suppose that $c\in V$ with $c\leq a,b$.
It will be sufficient to prove that $c\leq 0$. Since $ab=0$,
we have $aCb$ and $a\odot b=ab=0$. Clearly, $0\leq a\leq a-c+b$ and
$0\leq b\leq a+b-c$, whence by (iii), $0\leq a\odot(a-c+b)=a\odot
(a-c)+a\odot b=a\odot(a-c)=(a-c)\odot a$, and likewise $0\leq(b-c)
\odot b$. But $0\leq a-c\leq a$, so by (ii), $0\leq(a-c)\sqcap a
=a\sqcap(a-c)=a+(0\sqcap(-c))=a-(0\sqcup c)=a-c\sp{+}$, and we have
$0\leq c\sp{+}\leq a$. Similarly, $0\leq c\sp{+}\leq b$, and by
Lemma \ref{lm:TopL3} (i), $c\sp{+}=0$, so $c=-c\sp{-}\leq 0$.

\smallskip

Finally if $a\in V\Rightarrow|a|\in V$, then (iv) $\Rightarrow$ (v)
by Theorem \ref{th:PropR}.
\end{proof}

In \cite[\S 3]{TopVL}, condition (ii) of the following lemma is
called property SQ, and as mentioned above, condition (i) is
called property J, condition (iii) is called property P, and (iv)
is Ogasawara's condition O. The proof is adapted from \cite
[Proposition 6]{TopVL}. In the lemma, note that $a\odot b$,
$a\sp{2}$, and $b\sp{2}$ do not necessarily belong to $V$.

\begin{lemma} \label{lm:JSQP}
Suppose that $V$ satisfies the condition $a\in V\Rightarrow|a|
\in V$ and consider the following conditions for all $a,b\in
V${\rm:}
\begin{enumerate}
\item  $0\leq a,b\Rightarrow 0\leq a\odot b$.
\item $-b\leq a\leq b\Rightarrow a\sp{2}\leq b\sp{2}$.
\item $0\leq a,b\Rightarrow 0\leq a\sqcap b$.
\item $0\leq a\leq b\Rightarrow a\sp{2}\leq b\sp{2}$.
\end{enumerate}
Then {\rm(i)} $\Leftrightarrow$ {\rm(ii)} $\Leftrightarrow$ {\rm[(iii)}
and {\rm(iv)]}.
\end{lemma}

\begin{proof}
(i) $\Rightarrow$ (ii). Assume (i) and $-b\leq a\leq b$. Then $0\leq b-a,
b+a$, so $0\leq(b-a)\odot(b+a)=b\sp{2}-a\sp{2}$.

\smallskip

(ii) $\Rightarrow$ (i). Assume (ii) and $0\leq a,b$. Since (ii)
$\Rightarrow$ (iii), we have $|a-b|\leq a+b$. Since $|a-b|\in V$ and
(ii) $\Rightarrow$ (iv), it follows that $(a-b)\sp{2}=|a-b|\sp{2}
\leq(a+b)\sp{2}$, whence $a\odot b=\frac14[(a+b)\sp{2}-(a-b)\sp{2}]
\geq 0$.

\smallskip

(ii) $\Rightarrow$ (iii). Assume (ii) and $0\leq a,b$. Then $-(a+b)\leq
b-a\leq a+b$, so $(a-b)\sp{2}\leq(a+b)\sp{2}$, and by MSR, $|a-b|\leq
a+b$, whence $0\leq\frac12(a+b-|a-b|)=a\sqcap b$.

\smallskip

(ii) $\Rightarrow$ (iv). This is obvious.

\smallskip

[(iii) and (iv)] $\Rightarrow$ (ii).  Assume (iii), (iv), and $-b
\leq a\leq b$. Then $0\leq b-a, b+a$, so by (iii), $0\leq(b-a)
\sqcap(b+a)=b-|a|$, whence $|a|\leq b$ with $|a|\in V$. Thus, by
(iv), $a\sp{2}=|a|\sp{2}\leq b\sp{2}$.
\end{proof}

\begin{lemma} \label{lm:CondsforVL}
Suppose that $a\in V\Rightarrow |a|\in V$ and consider the following
conditions{\rm:}
\begin{enumerate}
\item $V$ is commutative.
\item If $a,b\in V$ and $a\leq b$, then $aCb$.
\item If $a,b\in V$ and $0,a\leq b$, then $a\sp{+}\leq b$.
\item If $a\in V$, then $a\vee\sb{V}0$ exists in $V$ and $a\sp{+}=
 a\vee\sb{V}0$.
\item If $a,b\in V$, then $a\vee\sb{V}b$ exists in $V$ and $a\vee\sb{V}b
 =(a-b)\sp{+}+b$.
\item $V$ is a vector lattice.
\item $E\cap V$ is an MV-effect algebra.
\item $P\cap V$ is commutative.
\end{enumerate}

Then {\rm(i)} $\Rightarrow$ {\rm(ii)} $\Rightarrow$ {\rm(iii)} $\Rightarrow$
{\rm(iv)} $\Rightarrow$ {\rm(v)} $\Rightarrow$ {\rm(vi)}. Moreover, if
$1\in V$, then {\rm(vi)} $\Rightarrow$ {\rm(vii)} $\Rightarrow$ {\rm(viii)}.
Finally, if $1\in V$ and $a\in V\Rightarrow a\dg\in V$, then conditions
{\rm(i)--\rm(viii)} are mutually equivalent.
\end{lemma}

\begin{proof} Assume that $a\in V\Rightarrow|a|\in V$. Obviously (i)
$\Rightarrow$ (ii).

\smallskip

(ii) $\Rightarrow$ (iii). Assume (ii) and let $a,b\in V$ with $0,a\leq b$.
By (ii), $aCb$. To prove (iii), we have to show that $a\sp{+}\leq b$. Put
$p:=(a\sp{+})\dg$. (We note that $p$ does not necessarily belong to $V$.)
Thus $pa\sp{+}=a\sp{+}$; $p\in CC(a\sp{+})$; and as $a\sp{+}\in CC(a)$, we
have $p\in CC(a)\subseteq C(a)$. Furthermore, as $a\sp{+}a\sp{-}=0$, it
follows that $pa\sp{-}=0$, whence $pa=p(a\sp{+}-a\sp{-})=pa\sp{+}=a\sp{+}$.
Also, as $aCb$ and $p\in CC(a)$, we have $pCb$, so $pC(b-a)$, and since $0
\leq p,b-a$, it follows that $0\leq p(b-a)$, whence $a\sp{+}=pa\leq pb$.
Moreover, as $bC(1-p)$ and $0\leq b,1-p$, we have $0\leq b(1-p)=b-bp$,
whence $bp\leq b$, and therefore $a\sp{+}\leq bp\leq b$.

\smallskip

(iii) $\Rightarrow$ (iv).  Assume (iii) and let $a\in A$. We have $a,0\leq
a\sqcup 0=a\sp{+}$. Moreover, by (iii), if $b\in V$ and $a,0\leq b$, we
have $a\sp{+}\leq b$, whence $a\sp{+}=a\vee\sb{V}0$.

\smallskip

(iv) $\Rightarrow$ (v). Assume (iv) and let $a,b\in V$. We claim that
$(a-b)\sp{+}+b$ is the supremum $a\vee\sb{V}b$ in $V$ of $a$
and $b$. As $a-b\leq(a-b)\sp{+}$, we have $a\leq(a-b)\sp{+}+b$. Clearly,
$b\leq(a-b)\sp{+}+b$. Suppose that $c\in V$ with $a,b\leq c$. We have to
show that $(a-b)\sp{+}+b\leq c$, i.e., that $(a-b)\sp{+}\leq c-b$. But
$a-b, 0\leq c-b$, and since $(a-b)\sp{+}$ is the supremum in $B$ of $a-b$
and $0$, it follows that $(a-b)\sp{+}\leq c-b$. Therefore $(a-b)\sp{+}+b=
a\vee\sb{V}b$.

\smallskip

(v) $\Rightarrow$ (vi). Assume (v) and let $a,b\in V$. Then the infimum
of $a$ and $b$ in $V$ is given by $a\wedge\sb{V}b=-((-a)\vee\sb{V}(-b))$,
whence $V$ is a vector lattice.

Now assume that $1\in V$.

(vi) $\Rightarrow$ (vii). Assume that $V$ is a vector lattice and put
$F:=E\cap V$. Clearly, $1\in F$ and $F$ is a sub-effect algebra of $E$, so
$F$ is an effect algebra in its own right under the restriction to $F$ of the
orthosummation $\oplus$ and the orthosupplementation $e\mapsto e\sp{\perp}=
1-e$ on $E$. Thus, for $e,f\in F$, $e\perp f$ in $F$ iff $e+f\leq 1$, in
which case $e\oplus f=e+f\in F$. If $e,f\in F$, then $0\leq e,f\leq 1$, so
$0\leq e\wedge\sb{V}f\leq e\vee\sb{V}f\leq 1$, whence $e\wedge\sb{V}f, e\vee
\sb{V}f\in F$, and therefore $F$ is a lattice effect algebra with $e\wedge
\sb{V}f=e\wedge\sb{F}f$ and $e\vee\sb{V}f=e\vee\sb{F}f\in F$.  Suppose that
$e,f\in F$ and $e\wedge\sb{F}f=e\wedge\sb{V}f=0$. By Theorem \ref{th:Topveclat}
(i), $e\sqcap f=e\wedge\sb{V}f=0$, $e\sqcup f=e\vee\sb{V}f$; hence $e+f=
e\sqcap f+e\sqcup f=e\sqcup f=e\vee\sb{V}f\leq 1$, and therefore $e\perp f$
in $F$. Thus, disjoint elements in $F$ are orthogonal, so by Theorem
\ref{th:JencaP2.1} and Definition \ref{df:MV,OML,Bool-EA}, $F$ is an MV-effect
algebra.

\smallskip

(vii) $\Rightarrow$ (viii). Assume (vii). Then $e$ is compatible with $f$
in $F=E\cap V$ for all $e,f\in F$ (Theorem \ref{th:JencaP2.1}). Clearly,
then, $e$ is compatible with $f$ in $E$ for all $e,f\in F$. In particular,
if $p,q\in P\cap V\subseteq F$, then $p$ is compatible with $q$ in $E$.
Therefore, by Lemma \ref{lm:compprojs}, $P\cap V$ is a commutative set.

Finally, assume that $1\in V$ and $a\in V\Rightarrow a\dg\in V$. Then
(viii) $\Leftrightarrow$ (i) by Lemma \ref{lm:J}.
\end{proof}

The following is our main theorem in this paper. In \cite[\S 3]{TopVL},
condition (viii) in the theorem is called property T. Related results
are obtained in \cite[Proposition 6, Theorem 2, and Corollary 9]{TopVL}.

\begin{theorem} \label{th:TopProp6}
Suppose that $1\in V$ and that $a\in V\Rightarrow |a|,\, a\dg\in V$. Then
the following conditions are mutually equivalent{\rm:}
\begin{enumerate}
\item $V$ is a vector lattice. {\rm(}Topping's property L{\rm)}
\item $V$ is commutative.
\item If $a,b\in V\sp{+}$ and $ab=0$, then $a\wedge\sb{V}b$ exists and
$a\wedge\sb{V}b=0$.
\item $a\in V\Rightarrow a\sp{+}=0\vee\sb{V}a$.
\item If $a,b,c\in V\sp{+}$ and $c\leq a+b$, then there exist
 $a\sb{1}, b\sb{1}\in V\sp{+}$ with $a\sb{1}\leq a$, $b\sb{1}\leq b$,
 and $c=a\sb{1}+b\sb{1}$.  {\rm(}RDP---Topping's property R{\rm)}
\item If $a,b\in V\sp{+}$, then $0\leq a\odot b$.  {\rm(}Topping's
 property J{\rm)}
\item If $a,b\in V$, then $-b\leq a\leq b\Rightarrow a\sp{2}\leq b\sp{2}$
 {\rm(}Topping's property SQ{\rm)}
\item If $a,b\in V$, then $|a+b|\leq|a|+|b|$. {\rm(}Topping's
 property T{\rm)}
\item If $a,b\in V$ and $-a\leq b\leq a$, then $|b|\leq a$.
\item If $a,b\in V\sp{+}$, then $0\leq a\sqcap b$. {\rm(}Topping's
 property P{\rm)}
\item If $a,b\in V$ and $0\leq a\leq b$, then $a\sp{2}\leq b\sp{2}$.
 {\rm(}Topping's property O{\rm)}
\item If $a,b\in V\sp{+}$, then $a\leq b\Rightarrow 0\leq a\odot b$.
 {\rm(}Topping's property J$\sp{+}${\rm)}
\end{enumerate}
\end{theorem}

\begin{proof}
Assume the hypotheses.

(i) $\Leftrightarrow$ (ii) $\Leftrightarrow$ (iii) $\Leftrightarrow$
(iv) $\Leftrightarrow$ (v) $\Leftrightarrow$ (vi) $\Leftrightarrow$
(vii).
That (i) $\Leftrightarrow$ (ii) $\Leftrightarrow$ (iv) follows from
the equivalence of parts (i), (iii), (iv), and (vi) of Lemma \ref
{lm:CondsforVL}. That (i) $\Leftrightarrow$ (iii) follows from
Theorem \ref{th:PropR}. That (i) $\Leftrightarrow$ (v) is a consequence
of Theorem \ref{th:PropR}, Lemma \ref{lm:J} yields (ii) $\Leftrightarrow$
(vi), and (vi) $\Leftrightarrow$ (vii) follows from Lemma \ref{lm:JSQP}.

\smallskip

(ii) $\Rightarrow$ (viii), (ix), (x), and (xi). Assume (ii).
Then (viix), (ix), (x), and (xi) follow from Lemmas \ref{lm:AbValProps}
(i), \ref{lm:AbValProps} (ii), \ref{lm:TopC2} (i), and Lemma
\ref{lm:sqrtprops} (iii), respectively.

\smallskip

(viii), (ix), (x) $\Rightarrow$ (iv). The arguments in the proof
of \cite[Proposition 6]{TopVL} work in the synaptic algebra $A$, but
for completeness, we reproduce them here. Thus, assume that $a,b\in V$
and $0,a\leq b$, i.e., $0\leq b, b-a$. We have to prove that each
condition (viii), (ix), (x) implies that $a\sp{+}\leq b$. We have
$a=b-(b-a)$, so if (vii) holds, then $|a|\leq|b|+|b-a|=2b-a$, whence
$a\sp{+}=\frac12(|a|+a)\leq b$. Also we have $-2b+a\leq a\leq 2b
-a$, so if (viii) holds, then $|a|\leq|2b-a|=2b-a$, and $a\sp{+}
=\frac12(|a|+a)\leq b$. Finally, if (ix) holds, then $0\leq b
\sqcap(b-a)=\frac12[b+b-a-|b-(b-a)|=b-\frac12(a+|a|)=b-a\sp{+}$,
and again $a\sp{+}\leq b$.

\smallskip

That (xi) $\Rightarrow$ (i) and (xi) $\Leftrightarrow$ (xii) follow
from Lemma \ref{lm:O&Jplus}.
\end{proof}

Now we focus on the case in which $A$ is commutative. Any commutative
synaptic algebra is a commutative, associative, Archimedean partially
ordered real linear algebra with order unit; it may be regarded as its
own enveloping algebra; it is a normed linear algebra under the
order-unit norm, and by the next theorem, its OML of projections is
a Boolean algebra.

\begin{theorem} \label{th:comsynalg}
The following conditions are mutually equivalent{\rm:}
\begin{enumerate}
\item The synaptic algebra $A$ is commutative.
\item Comparable elements in $A$ commute.
\item $A$ is a vector lattice.
\item $E$ is an MV-effect algebra.
\item Every pair of effects $e,f\in E$ is compatible in $E$.
\item Every pair of projections $p,q\in P$ is compatible in $P$.
\item $P$ is commutative.
\item $P$ is a Boolean algebra.
\end{enumerate}
\end{theorem}

\begin{proof}
(i) $\Leftrightarrow$ (ii) $\Leftrightarrow$ (iii)
$\Leftrightarrow$ (v) $\Leftrightarrow$ (vii) follows immediately
from Lemma \ref{lm:CondsforVL}.

(iii) $\Rightarrow$ (iv). Assume (iii). Then (i) also holds. Let $e,f
\in E$. Then we have $ef=fe$; moreover $0\leq e,f\leq 1$, so $0\leq e
\wedge\sb{V}f\leq e\vee\sb{V}f\leq 1$, whence $e\wedge\sb{V}f, \vee
\sb{V}f\in E$. Therefore $E$ is a lattice effect algebra with $e\wedge
\sb{E}f=e\wedge\sb{V}f$ and $e\vee\sb{E}f=e\vee\sb{V}f$. We claim that
$e$ and $f$ are compatible in $E$. By \cite[Lemma 2.7]{P&ESyn}, $ef=fe
\in $ E with $ef\leq e,f$. Therefore, $e-ef, f-ef, ef\in E$ with $(e-ef)
+(f-ef)+ef=e+f-ef=e+f(1-e)$. But $fC(1-e)$, so again by \cite[Lemma 2.7]
{P&ESyn}, $f(1-e)\in E$ with $f(1-e)\leq 1-e$, and therefore $(e-ef)+
(f-ef)+ef\leq e+(1-e)=1$. Thus, as $e=(e-ef)+ef$ and $f=(f-ef)+ef$, it
follows that $e$ and $f$ are compatible in $E$. Consequently, $E$ is
an MV-effect algebra, and we have (iii) $\Rightarrow$ (iv).

(iv) $\Rightarrow$ (v) $\Rightarrow$ (vi) $\Leftrightarrow$ (vii)
$\Leftrightarrow$ (viii). Obviously, (iv) $\Rightarrow$ (v), and (v)
$\Rightarrow$ (vi). That (vi) $\Leftrightarrow$ (vii) follows from
Lemma \ref{lm:compprojs} and that (vii) $\Leftrightarrow$ (viii) is
well-known in the theory of OMLs.

(vii) $\Leftrightarrow$ (i). If (vii) holds and $a,b\in A$, then the
projections in the spectral resolutions of $a$ and $b$ commute
with each other, whence $aCb$ by Theorem \ref{th:speccom} (ii). Thus,
(vii) $\Rightarrow$ (i), and the converse implication is obvious.
\end{proof}

In view of Theorem \ref{th:Topveclat} (i), Theorem \ref{th:comsynalg}
has the following immediate corollary.

\begin{corollary} \label{co:gensupinf}
If $A$ is comutative, then for $a,b\in A$, $a\wedge\sb{A}b=a\sqcap b$
and $a\vee\sb{A}b=a\sqcup b$.
\end{corollary}

A functional representation for commutative synaptic algebras can be
found in \cite[Theorem 4.1]{PSyn}.

\end{document}